\documentclass[reqno,12pt]{amsart}
\usepackage{geometry}
\usepackage{amsmath,amssymb,mathrsfs,color,stackengine}
\usepackage[colorlinks,
linkcolor=red,
anchorcolor=green,
citecolor=blue,
]{hyperref}
\usepackage[upint]{stix}
\usepackage{subfigure}
\usepackage{float}
\usepackage{times}
\usepackage{tikz}
\usetikzlibrary{intersections}
\usepackage{paralist}

\usepackage{comment}

\makeatletter
\def\setliststart#1{\setcounter{\@listctr}{#1}%
  \addtocounter{\@listctr}{-1}}
\makeatother

\makeatletter
\@addtoreset{figure}{section}
\makeatother

\usepackage{calc}
 \newtheorem{The}{Theorem}[section]
 \newtheorem{Cor}[The]{Corollary}
 \newtheorem{Lem}[The]{Lemma}
 \newtheorem{Pro}[The]{Proposition}
 \theoremstyle{definition}
 \newtheorem{defn}[The]{Definition}
 \newtheorem{Rem}[The]{Remark}
 
 \numberwithin{equation}{section}

\newcommand{\R}{\mathbb{R}}

\newcommand{\N}{\mathbb{N}}

\newcommand{\Lip}{\mbox{\rm Lip}\,}
\newcommand{\Liploc}{\mbox{$\mathrm{Lip}_{\mathrm{loc}}$}}

\title[Viscosity solution to timelike Lorentzian eikonal equation]{Viscosity solutions to a Cauchy type problem for timelike Lorentzian eikonal equation}
\author{Siyao Zhu, Xiaojun Cui \and Tianqi Shi}
\subjclass[2010]{53B30, 53C22, 49L25, 35F21}
\keywords{Lorentzian eikonal equation, globally hyperbolic space-time, viscosity solution, weak KAM theory}
\address{Department of Mathematics, Nanjing University, Nanjing 210093, China}
\email{zsydtc0807@163.com}
\address{Department of Mathematics, Nanjing University, Nanjing 210093, China}
\email{xcui@nju.edu.cn}
\address{Department of Mathematics, Nanjing University, Nanjing 210093, China}
\email{tqshi.math@gmail.com}

\begin{document}
\maketitle
\begin{abstract}
	In this paper, we propose a Cauchy type problem to the timelike Lorentzian eikonal equation on a globally hyperbolic space-time. For this equation, as the value of the solution on a Cauchy surface is known, we prove the existence of viscosity solutions on the past set (future set) of the Cauchy surface. Furthermore, when the time orientation of viscosity solution is consistent, the uniqueness and stability of viscosity solutions are also obtained.
\end{abstract}
\section{Introduction}
In the study of global Lorentzian geometry,  timelike Lorentzian eikonal equation
\begin{align*}
	g(\nabla u,\nabla u)=-1
\end{align*}
plays a vital role, where $g$ is the Lorentzian metric. The integral curves of a gradient field of a smooth solution, if exists, are timelike geodesics. Geometrically, many distance-like functions, such as Busemann functions, are essentially the solutions to timelike eikonal equation in some sense. Owing to this fact, distance-like functions are one of the key points to construct the compactification and completion of a manifold and to prove splitting theorem, even for finding the solutions to Einstein field equations. For more details, readers can refer \cite{beem_global_1981,flores_gromov_2013}.

It is well known an eikonal equation, as a particular kind of Hamilton-Jacobi equation, has no global smooth solutions in general. Actually, the expectable maximal regularity of these solutions is the local semiconcavity (see \cite{Cannarsa_Cheng3,Cannarsa_Sinestrari_book,Villani_book2009} for more information).
For the well-posedness of Hamilton-Jacobi equations, a more natural choice in some sense is the notion of viscosity solutions, systematically introduced by M-G. Crandall, L-C. Evans and P-L. Lions in \cite{Crandall_Evans_Lions1984,Crandall_Lions1983} during the early 1980s, although the discussion on well-posedness of evolutionary Hamilton-Jacobi equations
\begin{align*}
	\frac{\partial u}{\partial t}+H\left(x,\frac{\partial u}{\partial x}\right)=0,\,\,(t,x)\in\R^+\times M
\end{align*}
had been developed in \cite{Douglis1961,Kruzkov1975,Krylov_book1987} before the theory of viscosity solutions. Until now, many monographs have done rich work on this topic. Since late 1990s, for Tonelli Hamiltonian $H$ which includes the eikonal case, Fathi established weak KAM theory, which made the connection between Aubry-Mather theory (\cite{Mather1991,Mane1992,Mather1993}) and the theory of viscosity solution of stationary Hamilton-Jacobi equation
\begin{align*}
	H(x,Du(x))=c,\,\,x\in M,
\end{align*}
 on Riemannian manifold (\cite{Fathi2020,Fathi_book,Fathi1997_1,Fathi1997_2,Fathi1998_1,Fathi1998_2}). Cauchy and Dirichlet problems of Hamilton-Jacobi equations are partly mentioned in \cite{Fathi2012,CCMW2019}.

To our knowledge, the convexity (positive definiteness), coercivity  (even superlinearity) of Hamiltonians are needed essentially in the classical theory of viscosity solutions, Mather theory and weak KAM theory, but these properties cannot apply to the general space-time. Even so, there are also a few research papers on viscosity solutions to eikonal equation of various space-time in recent years. Following from \cite{andersson_cosmological_1998}, the negative regular cosmology time function is the viscosity solution of eikonal equation, which means a kind of viscosity solutions is found in a specific globally hyperbolic space-time (\cite{cui_negative_2014}). Moreover, in the case of class A torus, \cite{jin_global_2018} and \cite{suhr2011aubrymather} developed the well-posedness of eikonal equations, weak KAM theory and Aubry-Mather theory for corresponding geodesic flows. \cite{suhr2019aubry} is also a good reference for Aubry-Mather theory for compact space-time. More recently, the work in \cite{zhu2023global} dealt the related problems in globally hyperbolic space-time. Authors gave the existence of global-defined viscosity solutions and the local variational representation of a solution which has a consistent time-orientation.

  Based on the ideas above, we consider a Cauchy type problem of eikonal equation on a globally hyperbolic space-time. Main results of this paper are as follows.

   Suppose $(M,g)$ is a globally hyperbolic space-time, then there exists a unique viscosity solution of the following problem:
\begin{align}\label{pre-ecg}
\left\{
	\begin{array}{ll}
		g(\nabla u(x), \nabla u(x))=-1, & x\in I^-[\Gamma],\\
		u|_\Gamma=\phi\in\Liploc(\Gamma),&
	\end{array}
	\right.
\end{align}
with always past-directed directed time orientation, where $\Gamma$ is a Cauchy surface. Moreover, this solution can also be represented by a variational formula as
\begin{align*}
	u_\phi(x)=\inf_{y\in J_\Gamma^+(x)}\{\phi(y)-d(x,y)\}
\end{align*}
 and it yields that the solution is semiconcave. In the end, we show that \eqref{pre-ecg} continuously depends on the initial value function on fixed $\Gamma$.

The paper is organized as follows. In Section 2, we introduced some basic materials involving the Lorentzian geometry and causality theory, semiconcave functions and viscosity solutions. In Section 3, we answered affirmatively on the well-posedness of \eqref{pre-ecg} (Theorem \ref{lem7}, \ref{thm:unique} and \ref{thm:conti}). The general regularity and the variational formula of the solution were also given in Proposition \ref{pro:semiconcave} and Theorem \ref{lem7}.
\section{Preliminaries}
\subsection{Lorentzian causality theory}
First, we recall some basic concepts and facts on Lorentzian geometry mentioned in \cite{beem_global_1981,cui_viscosity_2016,cui_negative_2014,flores_gromov_2013,jin_global_2018}, as the setting of this paper.

Let $(M, g)$ be a \textit{space-time}, i.e., a connected, time-oriented, smooth Lorentzian manifold with the Lorentzian metric $g$, where the signature for $g$ is $(-,+, \cdots , +)$. A point $p\in M$ is usually called an \textit{event} from the viewpoint of general relativity. For each $p\in M$, we use $T_p M$ and $TM$ to denote the tangent space at $p$ and tangent bundle of $M$, respectively.

A  tangent vector $V \in T_p M$ is called \textit{timelike}, \textit{spacelike} or \textit{lightlike} if  $g(V, V)<0$, $g(V, V)>0$ or $g(V, V)=0$ ($V\neq0$), respectively. Also, a tangent vector $V\in T_p M$ is called \textit{causal} if $V$ is either timelke or lightlike and is called \textit{non-spacelike} if $g(V, V)\leq 0$. In other words, the set of \textit{non-spacelike} vectors consists of zero vector and causal vectors.

Thus, the property of time-orientation of $(M,g)$ ensures that $M$ admits a continuous and nowhere vanishing timelike vector field $X$, which is used to separate the causal vectors at each base point into \textit{past-directed} and \textit{future-directed} ones. More precisely, a causal tangent vector $V\in T_p M$ is said to be \textit{past-directed} (resp. \textit{future-directed}), if $g(X(p), V)>0$ (resp. $g(X(p), V)<0$).

Furthermore, a piecewise $C^1$ curve $\gamma: I\rightarrow M$ ($I\subset \R$ is an interval) is called causal, timelike, lightlike and past-directed or future-directed if the tangent vector $\dot{\gamma}(s)$ is causal, timelike, lightlike and past-directed or future-directed at every differentiable point $\gamma(s)$.

 For two events $p, q\in M$, if there exists a future-directed timelike
(resp. non-spacelike) curve from $p$ to $q$, we say $p$ and $q$ are \textit{chronologically related} (resp. \textit{causally related}) and denote by $p\ll q$ (resp. $p\leq q$). For any fixed $p\in M$ and $A\subset M$, it is common practice to define \textit{chronological past} $I^-(p), I^-[A]$, \textit{chronological future} $I^+(p), I^+[A]$, \textit{causal past} $J^-(p), J^-[A]$, and \textit{causal future} $J^+(p), J^+[A]$ as follows:
\begin{align*}
I^-(p)&:=\{q\in M: q\ll p\},\,I^-[A]:=\{q\in M: q\ll s ~~\text{for some}~~ s\in A\},\\
I^+(p)&:=\{q\in M: p\ll q\},\,I^+[A]:=\{q\in M: s\ll q ~~\text{for some}~~ s\in A\},\\
J^-(p)&:=\{q\in M: q\leq p\},\,\,J^-[A]:=\{q\in M: q\leq s ~~\text{for some}~~ s\in A\},\\
J^+(p)&:=\{q\in M: p\leq q\},\,\,J^+[A]:=\{q\in M: s\leq q ~~\text{for some}~~ s\in A\}.
\end{align*}
We call a subset $A$ of $M$ \textit{achronal} (resp. \textit{acausal}) if no two points of $A$ are chronologically related (resp. causally related).
 A space-time $(M, g)$ is said to be causal if there are no causal loops. A space-time $(M, g)$ is \textit{globally hyperbolic} if it is causal and the sets $J^+(p)\cap J^-(q)$ are compact for all $p, q\in M$.

 Throughout this paper, we always assume that $(M,g)$ is a globally hyperbolic space-time. Simultaneously, we are going to use two metrics on $M$: one is the Lorentzian metric $g$, the other is an auxiliary complete Riemannian metric $h$. The length functional, distance function, the norm and gradient associated to $g$ are denoted by $L(\cdot)$, $d(\cdot, \cdot)$, $|\cdot|$ and $\nabla$ respectively. Analogously, the corresponding length functional, distance function, norm and gradient associated to $h$ are denoted by $L_h(\cdot)$, $d_h(\cdot, \cdot)$, $|\cdot|_h$ and $\nabla_h$, respectively. More specifically, for a piecewise $C^1$ causal curve $\gamma: [0,1]\rightarrow M$, $L(\gamma)$ is given by
\begin{align*}
L(\gamma):=\int_0^1 \sqrt{-g(\dot{\gamma}(s), \dot{\gamma}(s))}\,ds.
\end{align*}
For any $x,y\in M$ with $x\leq y$, $d(x,y)$ is defined as
\begin{align*}
d(x,y):=\sup\left\{L(\gamma):\gamma \in \Omega_{x,y}\right\},
\end{align*}
where $\Omega_{x,y}$ denotes the path space of all future-directed causal and constant curves $\gamma:[0,1]\rightarrow M$ with $\gamma(0)=x$ and $\gamma(1)=y$.
 For a non-spacelike vector $V\in TM$, the associated Lorentzian norm is defined by $|V|:=\sqrt{-g(V,V)}$.

So far, the results on curvature in Lorentz geometry have been abundant. Here we only extract the following conclusions as the theoretical support of this paper, which will be revisited in Sections \ref{sec:existence} and \ref{sec:conti-depend}. For more details, we refer to \cite{alias_geometric_2010}.
\begin{Lem}[\protect{\cite[Lemma 3.1]{alias_geometric_2010}}]\label{lemLL}
Let $M^{n+1}$ be an $(n+1)$-dimensional spacetime such that the sectional curvatures of timelike planes of $M$ are bounded from above by some constant $c\in \R$. Assume that there exists $p\in M$ such that $\mathcal{I}^+(p)\neq \varnothing$, and let $q\in \mathcal{I}^+(p)$, (with $d(p, q)<\pi/ \sqrt{-c}$ when $c<0$). Then for every spacelike vector $V\in T_qM$ orthogonal to the gradient of $d(p, q)$ it holds that
\begin{align}\label{A20}
(\mathbf{Hess}\,d(p,q))\,(V,V)\geqslant -f_c(d(p,q))\langle V,V\rangle,
\end{align}
where $\mathbf{Hess}$ stands for the Hessian operator on $M$. When $c<0$ and $d(p, q )\geqslant \pi/\sqrt{-c}$, then it still holds that
\begin{align}\label{A21}
(\mathbf{Hess}\,d(p,q))\,(V,V)\geqslant -\frac{1}{d(p,q)}\langle V,V\rangle\geqslant-\frac{\sqrt{-c}}{\pi}\langle V,V\rangle.
\end{align}Here, the function $f_c$ above is defined as
\begin{align*}
 f_c(s):=\left\{
 \begin{array}{ll}
\sqrt{c}\coth(\sqrt{c}s), &c>0, s>0,  \\
 1/s, &c=0, s>0,  \\
\sqrt{-c}\cot(\sqrt{-c}s), &c<0, 0<s<\pi/\sqrt{-c}.
\end{array}\right.
\end{align*}
\end{Lem}

\subsection{Semiconvex functions and generalized gradient}
 \begin{defn}[\protect{\cite[Definition 2.6]{cui_negative_2014}}, Semiconvex/Semiconcave function]
 	Let $O$ be an open subset of $M$. A function $\psi: O\to\R$ is said to be \textit{semiconvex} (with constant $C$), if there exists  $C>0$ such that for any constant-speed geodesic path $\gamma(t)$, $t\in[0, 1]$, $\gamma(t)\in O$,
\begin{align}\label{E8}
         \psi(\gamma(t))\leq (1-t)\psi(\gamma(0))+t\psi(\gamma(1))+C\frac{t(1-t)}{2}d^2_h(\psi(\gamma(0)),\psi(\gamma(1))).
\end{align}A function $\psi:O\to\R$ is said to be \textit{semiconcave}, if $-\psi$ is semiconvex.
 \end{defn}
 A function $\psi: M\to\R$ is called \textit{locally semiconvex} if for each $p\in M$ there is a neighborhood $O$ of $p$ in $M$ such that (\ref{E8}) holds true as soon as $\gamma(t) \in O$ ($t\in [0, 1]$); or equivalently if (\ref{E8}) holds for some fixed positive number $C$ which only depends on compact subset $K$ of $O$.

 Readers can refer \cite{Cannarsa_Sinestrari_book} for more information on the classic definitions and properties of semiconcave functions.

 The following definitions and results in this subsection are well known when the gradient is induced by the Riemannian metric $h$. In the Lorentzian case, these can be guaranteed by bilinearity and nondegeneracy of $g$.
\begin{defn}[\protect{\cite[Definition 3.1.6]{Cannarsa_Sinestrari_book}}, Dini gradient] Assume that $f\in C(M)$. A vector $V \in T_pM$ is called to be a \textit{subgradient} (resp. \textit{supergradient}) of $f$ at $p\in M$, if there exist a neighborhood $O$ of $p$ and a $C^1$ function $\phi: O \rightarrow \R$, differentiable at $p$, with $f(p)=\phi(p)$, $\phi(p) \leqslant f(p)$ (resp. $\phi(p) \geqslant f(p)$) for every $x \in O$ and $\nabla \phi(p)= V$.
\end{defn}
Suoopse $f: M\to\R$ and we denote by $\nabla^- f(p)$ the set of subgradients of $f$ at $p$ and $\nabla^+ f(p)$ the set of supergradients of $f$ at $p$. It is easy to check that $\nabla^{\pm} f(p)$ are closed convex sets (possibly empty). $\nabla^-f(p)$ and $\nabla^+ f(p)$ are both nonempty if and only if $f$ is differentiable at $p$; in this case,
\begin{align*}
	\nabla^-f(p)=\nabla^+f(p)=\{\nabla f(p)\}.
\end{align*}

Let ($\Liploc(M)$) $\Lip(M)$ be the set of (locally) Lipschitz functions on $M$ with respect to the Riemannian metric $h$. For a locally Lipschitz function, as it is differentiable almost everywhere, the following generlized gradient makes sense.
\begin{defn}[\protect{\cite[Definition 3.1.10]{Cannarsa_Sinestrari_book}}, Reachable gradient]
	Let $f\in\Liploc(M)$, then a vector $V\in T_p M$ is called a \textit{reachable gradient} if there exists a sequence $\{p_k\}\subset M\setminus\{p\}$, such that $f$ is differentiable at $p_k$ for each $k\in\mathbb{N}$, and
	\begin{align}\label{eq: reachable gradient}
	\lim_{k\to\infty}p_k=p,\,\,\,\,\lim_{k\to\infty}\nabla f(p_k)=V,
	\end{align}
	where, the first limit in \eqref{eq: reachable gradient} is taken in the sense of manifold topology on $M$; the second one is taken in the sense of any fixed chart that contains $p$.
	
	The set of all reachable gradients of $\phi$ at $p$ is denoted by $\nabla^*f(p)$.
\end{defn}
\begin{Rem}
	Since the first limit is taken, we know that when $k$ is sufficiently large, $p_k$ goes into that chart. The second limit does not depend on the choice of chart.
\end{Rem}
\begin{Pro}[\protect{\cite[Proposition 3.3.4]{Cannarsa_Sinestrari_book}}]\label{pro:semiconcave}\label{pro:co}
	 If $\psi$ is locally semiconvex (resp. semiconcave) on $M$, then $\psi\in\Liploc(M)$ and we have the following properties:
	 \begin{enumerate}[\rm (1)]
	   	\item  $\nabla^-\psi(p)\neq\varnothing$ (resp. $\nabla^+\psi(p)\neq\varnothing$) for any $p\in M$;
	 	\item $\nabla^-\psi(p)=\mathrm{co}\nabla^*\psi(p)$ (resp. $\nabla^+\psi(p)=\mathrm{co}\nabla^*\psi(p)$) for any $p\in M$, where $\mathrm{co}$ means the convex hull of a set;
	 	\item $\nabla^-\psi(p)$ (resp. $\nabla^+\psi(p)$) is a singleton if and only if $\psi$ is differentiable at $p$;
	 	\item if $\{p_k\}$ converges to $p$ and $V_k\in\nabla^-\psi(p_k)$ (resp. $V_k\in \nabla^+\psi(p_k)$) converges to $V$, then $V\in\nabla^-\psi(p)$ (resp. $V\in\nabla^+\psi(p)$);
	 	\item if $\nabla^-\psi(p)$ (resp. $\nabla^+\psi(p)$) is a singleton for all $p\in M$, then $\psi\in C^1(M)$.
	 \end{enumerate}
\end{Pro}

The following lemma is a standard description about the gradients of a family of uniformly semiconcave functions and their limit functions.
\begin{Lem}[\protect{\cite[Lemma 2.1]{Cannarsa_Yu2009}}]\label{lem:uni-semiconcave}
	Let $\{\psi_k\}$ be a sequence of semiconcave functions on a bounded neighborhood $U$ with a uniform semiconcave constant $C$, $\{x_k\}\subset U$, $V_k\in\nabla^+\psi_k(x_k)$, $k\in\N$. If $x_k\to x$ and $\psi_k\rightrightarrows\psi$ as $k\to\infty$, then $\psi$ is semiconcave with constant $C$ and
	\begin{align*}
		\lim_{k\to\infty}d_E(V_k,D^+\psi(x))=0,
	\end{align*}where $d_E$ is the Euclidean distance. In particular, if $\psi_k$ is differentiable at $x_k$ and $\psi$ is differentiable at $x$, then $\lim_{k\to\infty}\nabla\psi_k(x_k)=\nabla\psi(x)$.
\end{Lem}

\subsection{Lorentzian eikonal equation and viscosity solution}
We recall the \textit{timelike  Lorentzian eikonal equation} with respect to the Lorentzian metric $g$:
\begin{align}\label{E1}\tag{E$_{g}$}
                 g(\nabla u,\nabla u)=-1.
\end{align}

Lorentzian eikonal equation plays an important role in the study of global Lorentzian geometry (see \cite{beem_global_1981, cui_negative_2014}). Similar to the Riemannian case, it is appropriate to introduce the concept of viscosity solution here, which could help ones better understand the behavior of timelike rays in the setting of Lorenzian geometry.
Following the idea originating from \cite[Definition 2]{Crandall_Evans_Lions1984}, we can state the definition of viscosity solution as below.
\begin{defn}[Viscosity solution of \eqref{E1}]
	$f\in C(M)$ is called a \textit{viscosity subsolution} of \eqref{E1} if for any $p\in M$, it satisfies
\begin{align*}
	             g(V, V)\leqslant -1,\,\,\,\,\,\forall\,V\in\nabla^+f(p).
\end{align*}
Similarly, $f\in C(M)$ is called a \textit{viscosity supersolution} of \eqref{E1} if for any $p\in M$, we have
\begin{align*}
             g(V, V)\geqslant -1,\,\,\,\,\,\forall\,V\in \nabla^- f(p).
\end{align*}
If $f$ satisfies both of the properties above, we call it a \textit{viscosity solution} of equation \eqref{E1}.
\end{defn}
We would naturally consider the following set
\begin{align*}
\mathscr{S}(M):=\left\{f\in\Liploc(M): f \mbox{ is global viscosity solution to \eqref{E1}}\right\}.
\end{align*}
%
In \cite{zhu2023global}, $\mathscr{S}(M)$ has been studied systematically. The first result is $\mathscr{S}(M)\neq\varnothing$, i.e., any globally hyperbolic space-time admits at least one globally defined locally Lipschitz viscosity solution. Moreover, by defining the time orientation of viscosity solution, the set $\mathscr{S}(M)$ can be divided into some subclasses. The definition of the time orientation is defined as follows.
\begin{defn}For $f\in\Lip(M)$ where $\nabla f$ is timelike as long as it exists:
\begin{enumerate}[\rm (1)]
	\item The time orientation of $f$ is called \textit{future-directed} at $x$, if any limiting gradient $V\in \nabla^* f(x)$ is  future-directed timelike; the time orientation of $f$ is \textit{past-directed} at $x$ if any limiting gradient $V\in \nabla^* f(x)$ is past-directed timelike;
\item We say the time orientation of $f$ changes at $x$ if there exist two timelike vectors $V, W$ in $\nabla^* f(x)$ such that $V$ is past-directed and $W$ is future-directed;
\item We call the time orientation of $f$ is consistent if the time orientation of $f$ is either future-directed at arbitrary $x\in M$ or past-directed at arbitrary $x\in M$. Otherwise, we say the time orientation of $f$ is non-consistent.
\end{enumerate}
\end{defn}
When the time orientation of viscosity solution of equation \eqref{E1} is consistent, we have the following properties in the view of weak KAM theory. Recall that a future-directed (resp. past-directed) inextendible timelike curve (see \cite[Page 61]{beem_global_1981}) $\gamma:[0,T)\to M$ is called a \textit{ray} if $d(\gamma(a),\gamma(b))=b-a$ (resp. $d(\gamma(b),\gamma(a))=b-a$) for all $0\leqslant a\leqslant b<T$.
\begin{The}[\protect{\cite[Theorem 3]{zhu2023global}}]\label{thm:1}
Suppose $u\in\mathscr{S}(M)$, the time orientation of which is always past-directed, then
\begin{enumerate}[\rm (1)]
	\item  $u$ is differentiable at $x\in M$ if and only if there exists a unique future-directed timelike ray $\gamma_x: [0, T)\rightarrow M$ that satisfies
\begin{align}\label{E15}
          \gamma_x(0)=x,\,\,\,u(\gamma_x(t))=u(x)-t,\,\,\,\,\forall\,t\in [0 ,T)
\end{align}
and
\begin{align}\label{E16}
         \dot{\gamma_x}(0)=-\nabla u(x),
\end{align}
where $T$ is the maximal existence time of $\gamma_x$;
  \item for any $x\in M$, $u$ is differentiable at any $\gamma_x(t)$ for $t\in (0, T)$ and $\dot{\gamma_x}(t)=-\nabla u(\gamma_x(t))$.
\end{enumerate}
\end{The}
Inspired by the work above, this paper continues to investiage the well-posedness of Cauchy type problem for a timelike Lorentzian eikonal equation. To pose the Cauchy type problem exactly, we introduce a preliminary concept.

\begin{defn}[Cauchy Surface]\label{Def2}
 A \textit{Cauchy surface} $\Gamma$ is a subset of $M$ for which every inextendible causal curve intersects exactly once.
\end{defn}

By \cite[Theorem 3]{bernard_lyapounov_2018}, there exists a \textit{steep temporal function} $\tau : M \to \R$ such that $\tau$ is strictly increasing along any future-directed causal curve and
\begin{equation}\label{E2}
|d \tau (V)|\geqslant \max\left\{\sqrt{-g(V, V)},|V|_h\right\}
\end{equation}
for all timelike vector $V \in TM$. Moreover, Corollary 1.8 in \cite{bernard_lyapounov_2018} implies that each level set $\tau_s:=\{p\in M| \tau(p)=s\}$  is a Cauchy surface.

Hereafter, we always use $\Gamma\subset M$ to represent a given Cauchy surface.
Now one can consider a \textit{Cauchy type problem for the timelike Lorentzian eikonal equation} on $M$ as
\begin{align}\label{A1}\tag{EC$_g$}
	\left\{
	\begin{array}{ll}
		g(\nabla u(x), \nabla u(x))=-1, & x\in I^-[\Gamma],\\
		u(x)=\phi(x),& x\in \Gamma,
	\end{array}
	\right.
\end{align}
where $\phi\in\Liploc(\Gamma)$.

\section{Well-posedness of \eqref{A1}}
\subsection{Existence of the viscosity solutions}\label{sec:existence}
For brevity, as $x\in I^-[\Gamma]$, we set
\begin{align*}
    I^+_{\Gamma}(x)&:=\Gamma\cap I^+(x)=\{y\in\Gamma:x\ll y\},\\
	J^+_{\Gamma}(x)&:=\Gamma\cap J^+(x)=\{y\in\Gamma:x\leq y\}.
\end{align*}
Suppose $\phi\in\Lip(\Gamma)$. For any $x\in I^-[\Gamma]$, consider the function
\begin{align}\label{eq:u_phi}
u_\phi(x):=\inf_{y\in J^+_{\Gamma}(x)}\{\phi(y)-d(x,y)\}.
\end{align}
The following lemma establishes the existence of minimizer of equation \eqref{eq:u_phi}.
\begin{Lem}
	\label{lem1}
For any $x\in I^-[\Gamma]$, there exists $y_x\in \Gamma$ such that
\begin{align}\label{E4}
u_\phi(x)=\phi(y_x)-d(x,y_x).
\end{align}
In other words, $y_x$ is the minimizer of \eqref{eq:u_phi} and the infimum is actually a minimum.
\end{Lem}
\begin{proof}
	We reparametrize all future-directed inextendible causal geodesics by $h$-arclength, then these geodesics satisfy a system of second order ordinary differential equations. Since $\Gamma$ is a Cauchy surface, any future-directed  inextendible causal curve $\gamma$ with $\gamma(0)=x$ must strike $\Gamma$ at a unique point. Therefore, we can define a map $T_{\Gamma,x}$ depends on $x$ as
	\begin{align*}
		T_{\Gamma,x}:U_hT_xM\cap\mathscr{C}_x&\to\Gamma\\
		V&\mapsto \gamma_{x,V}\cap\Gamma,
	\end{align*}where $U_hT_xM$ is the unit tangent sphere at $x$ with respect to Riemannian metric $h$, $\mathscr{C}_x$ is the future-pointing light cone at $x$ and $\gamma_{x,V}$ is a future-directed  inextendible causal curve with $\gamma(0)=x$ and $\dot\gamma(0)=V$. Applying the continuous dependence on the initial condition of ODE, it is easy to see that $T_{\Gamma,x}$ is continuous. Since the $U_hT_xM\cap \text{int}\,(\mathscr{C}_x)$ is diffeomorphic to an open ball in Euclidean space,
\begin{align*}
	J^+_{\Gamma}(x)=T_{\Gamma,x}(U_hT_xM\cap\mathscr{C}_x)=\overline{T_{\Gamma,x}(U_hT_xM\cap \mathrm{int}\,(\mathscr{C}_x))}=\overline{I^+_{\Gamma}(x)}
\end{align*}
is compact. Then there must be a point $y_x$ which minimizes $\phi(\cdot)-d(x,\cdot)$.
\end{proof}
\begin{defn}[Forward calibrated curve]
	For any fixed $u\in\mathscr{S}(M)$, a future-directed timelike curve $\alpha:[0,t]\rightarrow M$ is said to be a \textit{forward calibrated curve} of $u$ if $u(\alpha(0))=u(\alpha(s))-s$ for any $s\in [0,t]$.
\end{defn}
For any $x\in I^-[\Gamma]$, we always use $y_x$ to denote the minimizer of \eqref{eq:u_phi} . Moreover, from \eqref{E4}, we can obtain a future-directed timelike curve ${\gamma}_x$ with ${\gamma}_x(0)=x$, ${\gamma}_x(d(x,y_x))=y_x$ and
\begin{align}\label{A5}
u_\phi({\gamma}_x(t))-u_\phi({\gamma}_x(0))=t
\end{align}
for any $t\in [0, d(x, y_x)]$. Indeed, ${\gamma}_x$ is a forward calibrated curve for $u_\phi$ by definition.
In the rest of this paper, we always use $y_x$ and ${\gamma}_x$ to denote a minimizer of equality \eqref{eq:u_phi} and forward calibrated curve of $u_\phi$ at $x$.

With the help of Lemma \ref{lem1}, we can study the regularity of $u_\phi$. Actually, we will show that $u_\phi$ is a locally semiconcave function on $I^-(\Gamma)$.  As a preparation, we need to introduce the following concepts.
\begin{defn}
	[Upper support function]\label{Def2}
 Let $x\in I^-[\Gamma]$ and $\gamma$ be a future-directed timelike maximal geodesic with $\gamma(0)=x$ and $\gamma(d(x,y_x))=y_x$. For any fixed $p\in \gamma$ and a suitable neighborhood $O$ of $x$ satisfying $O \subset I^-(p)$, we call the function
\begin{align}\label{E5}
u_{\phi,x}: O&\to [-\infty,+\infty] \nonumber\\
z&\mapsto\phi(y_x)-d(z,p)-d(p,y_x).
\end{align}
an \textit{upper support function} of $u_\phi$ at $x$.
\end{defn}
\begin{Rem}
	Actually, for any choice of $p\in\gamma$,
	\begin{align*}
		u_{\phi,x}(z)&=\phi(y_x)-d(z,p)-d(p,y_x)\geqslant \phi(y_x)-d(z, y_x)\geqslant \phi(y_z)-d(z, y_z)=u_\phi(z)
	\end{align*}
	for any $z\in O$ and the equality holds when $z=x$. Readers may wonder about the dependence of $u_{\phi,x}$ on $p$ here, but the choice of $p$ does not prevent the following conclusions.
\end{Rem}
Relying on the upper support function, we can study the semiconcavity of $u_\phi$.
\begin{Lem}[\protect{\cite[Lemma 3.2]{andersson_cosmological_1998}}]\label{lem9}
	Let $U\subset\R^n$ be an open convex domain and $f\in C(U)$. Also, assume that there exists some $C>0$ such that for any $p\in U$, $f_p$, the upper support function of $f$ at $p$, satisfies $D^2f_p\leqslant CI$ in its domain of definition, then $f-\frac{C}{2}\|\cdot\|_E^2$ is concave in $U$, which means $f$ is semiconcave with constant $C$, where $\|\cdot\|_E$ denotes the Euclidean norm on $\R^n$.
\end{Lem}
\begin{Pro}\label{pro:semiconcave}
	$u_\phi$ is locally semiconcave on $I^-[\Gamma]$.
\end{Pro}
\begin{proof}
	Since $\gamma_x$ maximizes the distance between any pair of its points, there is no cut point to $x$ on $\gamma_x|_{(0, d(x,p))}$ for appropriately selected $p$. Then by \cite[Proposition 9.29]{beem_global_1981}, there exists a neighborhood $O$ of $x$ such that $\bar O$ is compact, $\bar O\varsubsetneqq I^-(p)$ and $d_p(\cdot):=d(\cdot,p)$ is a smooth function on $O$.
	
	Combining with Lemma \ref{lem9}, we only need to prove that there exists a constant $C(O)>0$ depending only on $O$ such that $D^2u_{\phi,x}\leqslant C(O)I$ near $x$ in the setting of Euclidean coordinates. By the definition of $u_{\phi,x}$, this is equivalent to prove that $D^2 d(\cdot, p)\geqslant -C(O)I$ near $x$ on $O$.
	
	Now for any $z\in O$, we denote $\eta_z$ the unique speed future-directed maximal geodesic connecting $z$ to $p$. In the following we will use the standard comparison principle to give an estimation of Hessian of Lorentzian distance function $d_p$ in terms of upper and lower bounds of timelike sectional curvatures of 2-planes containing $\dot{\eta}_z(0)$ and the length of maximal segment between $z$ and $p$, where the Hessian of $d_p$ is defined in terms of the Levi-Civita connection with respect to $g$.

Since $\bar O$ is a compact subset of $M$,
\begin{align*}
	\inf_{z\in\bar O,V\in \partial\mathscr C_z^{+1}}\left\{\bigg|\frac{\dot\eta_z(0)}{|\dot\eta_z(0)|_h}-V\bigg|_h\right\}>0,
\end{align*}
where $\mathscr C_z^{+1}:=\{V\in T_zM:V \text{ is a future-directed lightlike vector and } |V|_h=1\}$, which implies there exists a compact subset $K_O\subseteq TM$ such that $\cup_{z\in O}\{(\eta_z(0),\dot{\eta}_z(0))\}$ (thus $\cup_{z\in O}\{(\eta_z(t),\dot{\eta}_z(t)):t\in[0,d(z,p)]\}$) and $\cup_{z\in O}\{(z,V)\in TM:g(V,\dot\eta_z(0))=0, |V|=1\}$ are contained in $K_O$. Therefore, both the timelike sectional curvatures of 2-planes containing $\dot\eta_z(0)|_{z\in O}$ and the length of $\eta_z|_{[0,d(z,p)]}$ are bounded from above by some constant $C>0$ (e.g. \cite[Proposition 3.1]{andersson_cosmological_1998}). By Lemma \ref{lemLL}, there exists a constant $C_1(O)$ depending only on $O$ such that
\begin{align}\label{A22}
D^2 d(\cdot,p)\geqslant C_1(O)I.
\end{align}
For any $z\in O$, let $\{e_1, e_2,..., e_{n-1}, \dot{\eta}_z(0)\}$ be a local orthonormal frame of $T_zM$. Then we have
\begin{align*}
\frac{\partial^2d_p(z)}{\partial x_i \partial x_j}-\Gamma_{ij}^k \frac{\partial d_p(z)}{\partial x_k}\geqslant C_1(O)I,
\end{align*}
where $\Gamma_{ij}^k$ is the Christoffel symbol. Due to the smoothness of $\Gamma_{ij}^k$ and $d_p$ on $O$, we can get another constant $C_2(O)$, such that
\begin{align}\label{A23}
\frac{\partial^2d(z,p)}{\partial x_i \partial x_j}\geqslant C_2(O)I,
\end{align}
for any $z\in O$. From the estimate on Hessian of $d_p$ in Euclidean setting and Lemma \ref{lem9}, we have the local semiconcavity of $u_\phi$ on $I^-[\Gamma]$.
\end{proof}

\begin{Lem}[\protect{\cite[Lemma 2.5]{cui_negative_2014}}]
	\label{lem4}
  Suppose $V \in T_p M$ is a nonzero past-directed non-spacelike vector and
  \begin{align*}
  	g(V, W)\geqslant \sqrt{-g(W, W)}
  \end{align*}
   for any future-directed causal vector $W$, then $V$ is not lightlike.
\end{Lem}
\begin{Pro}\label{lem2}
 If $u_\phi$ is differentiable at $x$, then $\nabla u_\phi(x)$ is a past-directed timelike vector. Moreover, in this case, $u_\phi$ satisfies \eqref{E1} at $x$.
\end{Pro}
\begin{proof}
	Assume that $u_\phi$ is differentiable at $x\in I^-[\Gamma]$. Choosing any fixed smooth future-directed causal curve $\gamma: [0, T)\to M$ with some small $T>0$ and $\gamma(0)=x$, by reverse triangle inequality and the definition of $u_\phi$, we have
\begin{align*}
u_\phi(\gamma(t))-d(\gamma(0), \gamma(t))&=\min_{y\in J_\Gamma^+(\gamma(t))}\{\phi(y)-d(\gamma(t), y)\}-d(\gamma(0), \gamma(t))\\
&\geqslant\min_{y\in J_\Gamma^+(\gamma(t))}\{\phi(y)-d(\gamma(0), y)\}\\
&\geqslant\min_{y\in J_\Gamma^+(\gamma(0))}\{\phi(y)-d(\gamma(0), y)\}=u_\phi(\gamma(0))
\end{align*}since $J_\Gamma^+(\gamma(t))\subset J_\Gamma^+(\gamma(0))$ for every $t\in [0,T)$. Armed with that information,
\begin{align*}
u_\phi(\gamma(t))-u_\phi(\gamma(0))&\geqslant d(\gamma(0),\gamma(t))
\geqslant \int_0^t \sqrt{-g(\dot{\gamma}(s), \dot{\gamma}(s))}\,ds
\end{align*}
for every $t\in [0, T)$. Dividing by $t$ on both sides, we get
\begin{equation}\label{E2}
                 \frac{u_\phi(\gamma(t))-u_\phi(x)}{t}\geqslant\frac{1}{t}\int_0^t \sqrt{-g(\dot{\gamma}(s), \dot{\gamma}(s))}\,ds.
\end{equation}
As $t\searrow 0$, the differentiability of $u_\phi$ at $x$ and inequality (\ref{E2}) yield
\[
d u_\phi(x)(\dot{\gamma}(0))\geqslant \sqrt{-g(\dot{\gamma}(0), \dot{\gamma}(0))}.
\]
Note that $du_\phi(x)(\dot{\gamma}(0))=g(\nabla u_\phi(x), \dot{\gamma}(0))$, this means
\begin{equation}\label{E3}
                 g(\nabla u_\phi(x), \dot{\gamma}(0))\geq \sqrt{-g(\dot{\gamma}(0), \dot{\gamma}(0))}
\end{equation}
for any future-directed causal vector $\dot{\gamma}(0)$.

We now prove that $\nabla u_\phi(x)$ is timelike. First, it is easy to see that $\nabla u_\phi(x)\neq 0$.
Furthermore, suppose $\nabla u_\phi(x)$ is spacelike, then there exists a smooth future-directed timelike curve $\bar{\gamma}: [0, \varepsilon) \to M$ such that $\bar{\gamma}(0)=p$ and $g(\dot{\bar{\gamma}}(0), \nabla u_\phi(x))=0$, which means that $\nabla u_\phi(x)$ and $\dot{\bar{\gamma}}(0)$ are orthogonal with respect to the Lorentzian metric $g$. On the other hand, from (\ref{E3}), we get $g(\nabla u_\phi(x), \dot{\bar{\gamma}}(0))>0$, which contradicts the orthogonal hypothesis. Thus, $\nabla u_\phi(p)$ is a causal vector.  In addition,  $g(\nabla u_\phi(x), V)\geqslant 0$ for any future-directed vector $V\in T_pM$. This implies that $\nabla u_\phi(x)$ is past-directed. By Lemma \ref{lem4}, we have $\nabla u_\phi(x)$ is indeed a past-directed timelike vector.

For the rest of the proof, we can show that $|\nabla u_\phi(x)|=1$ equivalently. First we choose a future-directed smooth causal curve $\eta: [0, \varepsilon)\rightarrow M$ with $\eta(0)=x$, $\dot{\eta}(0)=-\nabla u_\phi(x)$. Then by inequality (\ref{E3}),
\begin{align*}
|\nabla u_\phi(x)|^2=-g(\nabla u_\phi(x), \nabla u_\phi(x))=g(\nabla u_\phi(x), \dot{\eta}(0))\geqslant\sqrt{-g(\dot{\eta}(0), \dot{\eta}(0))}=|\nabla u_\phi(x)|.
\end{align*}
Hence, $|\nabla u_\phi(x)|\geqslant 1$ since $\nabla u_\phi(x)\neq 0$.
Recall that for any $x\in I^-[\Gamma]$ there exists a $y_x\in \Gamma$ and a forward calibrated curve ${\gamma}_x: [0,d(x,y_x)]\to M$ that starts at $x$ and ends at $y_x$ satisfies
\begin{align*}
	u_\phi({\gamma}_x(t))-u_\phi({\gamma}_x(0))=t
\end{align*}
for any $t\in [0, d(x, y_x)]$. Then we get
\begin{align*}
\frac{u_\phi({\gamma}_x(t))-u_\phi({\gamma}_x(0))}{t-0}=1
\end{align*}
for every $t\in [0, d(x, y_x)]$. Let $t\searrow 0$, and by the the reverse Cauchy-Schwarz inequality for causal vectors (\cite[Section 2.2.1]{sachs_general_1977}) and ${\gamma}_x$ is a timelike (unit-speed) curve,  we have
\begin{equation}\label{E7}
  1=d u_\phi(x)(\dot{{\gamma}}_x(0))=g(\nabla u_\phi(x), \dot{{\gamma}}_x(0))\geqslant |\nabla u_\phi(x)| |\dot{{\gamma}}_x(0)|=|\nabla u_\phi(x)|.
\end{equation}
 Combining with $|\nabla u_\phi(x)|\geqslant 1$, \eqref{E7} implies $|\nabla u_\phi(x)|=1$.
\end{proof}
\begin{The}\label{lem7}
  Let $(M,g)$ be a globally hyperbolic space-time, then $u_\phi$ is a viscosity solution to \eqref{E1} and \eqref{A1} on $I^-[\Gamma]$.
\end{The}
\begin{proof}
	By the definition of viscosity solution, we need to show $u_\phi$ is the subsolution and supersolution of equation (\ref{E1}). Firstly, for any $V\in \nabla^+ u_\phi(x)$,
by Lemma \ref{pro:co},
\begin{align*}
V\in \nabla^+ u_\phi(x)=\mathrm{co}\nabla^*u_\phi(x).
\end{align*}
Together with Lemma \ref{lem2}, we can conclude that $g(V,V)\leqslant -1$. It means that $u_\phi$ is a subsolution of the equation (\ref{E1}).

On the other hand, since $u_\phi$ is locally semiconcave, $u_\phi$ is differentiable at $x\in M$ if and only if $\nabla^- u_\phi(x)\neq \varnothing$ and $\nabla^+ u_\phi(x)=\nabla^- u_\phi(x)=\{\nabla u_\phi(x)\}$. Thus, Proposition \ref{lem2} implies that $u_\phi$ is a supersolution. Therefore $u_\phi$ is indeed a viscosity solution of the equation \eqref{E1}. Moreover, the definition of $u_\phi$ promises that $u_\phi|_{\Gamma}=\phi$, which means $u_\phi$ is also a viscosity solution to \eqref{A1}.
\end{proof}
\subsection{Uniqueness of the solution}\label{sec:uniqueness}
For brevity of the following illustration, we first define some subsets of $\mathscr{S}(M)$ below:
\begin{align*}
	\mathscr{S}_{\Gamma_\phi}(M)&:=\{u\in \mathscr{S}(M): u|_\Gamma=\phi\},\\
	\mathscr{S}^-_{\Gamma_\phi}(M)&:=\{u\in \mathscr{S}(M): u|_\Gamma=\phi\text{ and the time orientation of }u\text{ is always past-directed}\}.
\end{align*}
And about \eqref{A1}, we also define
\begin{align*}
	\mathscr{C}_\phi(I^-[\Gamma])&:=\left\{f\in\Liploc(M): f \mbox{ is global viscosity solution to \eqref{A1}}\right\},\\
	\mathscr{C}_\phi^-(I^-[\Gamma])&:=\left\{f\in\mathscr{C}_\phi(I^-[\Gamma]): \text{ the time orientation of }u\text{ is always past-directed}\right\}.
\end{align*}
Note that we have no reason here to claim that $\mathscr{S}_{\Gamma_\phi}(M)=\mathscr{C}_\phi(I^-[\Gamma])$ and $\mathscr{S}_{\Gamma_\phi}^-(M)=\mathscr{C}_\phi^-(I^-[\Gamma])$.
\begin{Lem}\label{lem:unique}
	Suppose $\Gamma$ is a Cauchy surface of $(M,g)$ and $\phi\in\Liploc(\Gamma)$,  then
	\begin{align}\label{eq:unique1}
		\left\{u|_{I^-[\Gamma]}:u\in\mathscr{S}_{\Gamma_\phi}^-(M)\right\}=\left\{u_\phi\right\},
	\end{align}i.e., in some sense, the solutions in $\mathscr{S}_{\Gamma_\phi}^-(M)$ are unique on half-space $I^-[\Gamma]$.
\end{Lem}
\begin{proof}
	Let $u\in\mathscr{S}^-_{\Gamma_\phi}(M)$. Recall that $u_\phi$ has a forward calibrated curve $\gamma_x$ from $x$ with unit speed. By the definition of $u_\phi$ and $u$, for $x\in I^-[\Gamma]$, we have $u(x)\leqslant u({\gamma}_x(t))-d(x,{\gamma}_x(t))$ for any $t\in [0, d(x, y_x)]$, and when $t=d(x,y_x)$,
\begin{align*}
  u(x)&\leqslant u(y_x)-d(x,y_x)=\phi(y_x)-d(x,y_x)=u_\phi(x).
\end{align*}
On the other hand, Theorem \ref{thm:1}
 claims $u$ has an inextendiable forward calibrated curve $\eta_x$ from any fixed $x\in M$. There exists some $T>0$ such that $\eta_x(T)\in\Gamma$. Then, we have
\begin{align*}
  u(x)=u(\eta_x(T))-T=\phi(\eta_x(T))-d(x,\eta_x(T))\geqslant u_\phi(x),
\end{align*}which means $u(x)=u_\phi(x)$. Arbitrariness of $u$ and $x$ implies \eqref{eq:unique1} holds true.
\end{proof}
\begin{Rem}
	It is worth noting that the condition on consistence of time orientation of solutions is necessary, and the uniqueness of viscosity solution of \eqref{A1} doesn't hold true any more in general. To be precise, see the following counterexample.

Consider the 2-dimensional Minkowski space-time $(\R^2, dy^2-dx^2)$, let
\begin{align*}
	\Gamma=\{(x,y)\in\R^2:x=0\},\,\,\,\, \phi=0.
\end{align*}
Then for any $c<0$, $u_c(x,y):=|x-c|+c$ is the solution to system \eqref{A1}.
\end{Rem}
\begin{Cor}\label{cor:forward calibrated}
Following from Theorem \ref{thm:1} and Lemma \ref{lem:unique}, we have
	\begin{enumerate}[\rm (1)]
		\item For any $u\in \mathscr{S}^-_{\Gamma_\phi}(M)$, $u$ is differentiable at $x\in M$ if and only if there exists a unique forward calibrated curve $\gamma_x: [0, d(x, \Gamma)]\rightarrow M$ that satisfies
          \begin{align}\label{E12}
           \gamma_x(0)=x,~~ u(\gamma_x(t))=u(x)-t~~\mbox{for every } t\in [0 ,d(x,\Gamma)]
          \end{align}
          and $ \dot{\gamma_x}(0)=-\nabla u(x)$;
       \item  For any $u\in\mathscr{S}^-_{\Gamma_\phi}(M)$ and any $x\in M$, $u$ is differentiable at any $\gamma_x(t)$ for $t\in (0, d(x,\Gamma))$ and $\dot{\gamma_x}(t)=-\nabla u(\gamma_x(t))$.
	\end{enumerate}
\end{Cor}
\begin{The}[\protect{\cite[Theorem 2]{zhu2023global}}]\label{the7}
 Assume that $u\in\mathscr{S}(M)$ and the time orientation of $u$ is always past-directed, then there exists some $t_0>0$, which depends on $x$, such that:
\begin{enumerate}[\rm (1)]
  \item for each $s\in \mathrm{Image}\,(u)$, $u^{-1}(s)$ is a partial Cauchy surface;
  \item for any $t\in(0,t_0)$,
\begin{align*}
           u(x)=\inf\limits_{x\leq y, d(x,y)=t}\{u(y)-t\};
\end{align*}
  \item $u$ is locally semiconcave on $(M,g)$.
\end{enumerate}
\end{The}
\begin{The}\label{thm:unique}
Suppose $\Gamma$ is a Cauchy surface of $(M,g)$ and $\phi\in\Liploc(\Gamma)$,  then
	\begin{align}\label{eq:unique2}
		\mathscr{C}_\phi^-(I^-[\Gamma])=\left\{u_\phi\right\}.
	\end{align}In other words, $u_\phi$ is the unique viscosity solution to $\eqref{A1}$ with always past-directed time orientation.
\end{The}
\begin{proof}
By the proof of Theorem \ref{the7}, if the time orientation of $u$ is always past-directed, $u$ has a local variational representation. Thus, we still have a forward calibrated curve $\eta_x:[0, t_0]\rightarrow M$ of $u$ at $x$. By \cite[Subsection 5.2]{zhu2023global}, we can extend this curve $\eta_x$ until $\eta_x$ intersects $\Gamma$. Then the rest of proof follows from Lemma \ref{lem:unique}.
\end{proof}
\subsection{The continuous dependence of solutions \eqref{A1} on initial values}\label{sec:conti-depend}
Suppose $\phi_n\in\Liploc(\Gamma)$ and consider the system
\begin{align}\label{A1n}\tag{EC$_{g,n}$}
\left\{
\begin{array}{ll}
	g(\nabla u_n(x),\nabla u_n(x))=-1,&x\in I^-[\Gamma],\\
	u_n(x)=\phi_n(x),&x\in\Gamma.
\end{array}	
\right.
\end{align}As a result of two subsections above, $u_n:=u_{\phi_n}$ is the unique viscosity solution to \eqref{A1n}. To study the continuous dependence of system \eqref{A1} on initial values, i.e., in some sense, we hope
\begin{align*}
	\lim_{n\to\infty}\eqref{A1n}=\eqref{A1}.
\end{align*}
under the assumption $\phi_n\rightrightarrows\phi$ as $n\to\infty$ on $\Gamma$.

\begin{Lem}\label{lem:equi-lip-semi}
	Suppose $\phi_n\rightrightarrows\phi$ on $\Gamma$. Then $\{u_n\}_{n\geqslant 1}$ has a subsequence of locally equi-Lipschitz and equi-semiconcave functions.
\end{Lem}
\begin{proof}
	Owing to Lemma \ref{lem9}, the lemma is equivalent to the claim that for a suitable neighbourhood $O$ of $x$, the bound of the Hessian of distance function $d(\cdot,p_n)$ defined on $O$ is independent on $n$ for some $p_n\in I^-[\Gamma]$. Recall that for each $n$ there is a forward calibrated curve denoted by $\gamma_x^n:[0, d(x, y_x^n)]\rightarrow M$ connecting $x$ to $y_x^n$, where $y_x^n$ is a minimizer for $\phi_n(\cdot)-d(x,\cdot)$.
	
	For every $n$, we can also extend $\gamma_x^n$ to a future-directed inextendible curve and reparameter by $h$-arc length as $\tilde{\gamma}_x^n:[0, +\infty)\rightarrow M$. \cite[Lemma 14.2]{beem_global_1981} implies that there exists a subsequence $\{\tilde{\gamma}_x^{n_k}\}$ and an intextendible future-directed causal curve $\tilde{\gamma}_x:[0, +\infty)\rightarrow M$ with $\tilde{\gamma}_x(0)=x$ such that $\{\tilde{\gamma}_x^{n_k}\}$ converges to $\tilde{\gamma}_x$ uniformly under the Riemannian metric $h$ on arbitrary compact subset of $\R^+$.
	
	Since $\Gamma$ is a Cauchy surface, there exists $T>0$ such that $\tilde{\gamma}_x(T)\in\Gamma$. Recall that $\gamma_x^n$ is a maximal segment on $[0, d(x, y_x^n)]$. There exists $T^*<T$ and $\{T_{n_k}\}$, which satisfy $\gamma_x^{n_k}(T_{n_k})=\tilde{\gamma}_x^{n_k}(T^*)$ and $\gamma_x^{n_k}$ are well-defined on $[0, T_{n_k}]$. By the upper semicontinuity of Lorentzian length functional, we have
\begin{align*}
	L(\tilde{\gamma}_x)|_{[0, T^*]}\geqslant\limsup_{k\rightarrow\infty}L(\tilde{\gamma}_x^{n_k})|_{[0, T^*]}=\limsup\limits_{k\rightarrow\infty}L(\gamma_x^{n_k})|_{[0, T_{n_k}]}=\limsup\limits_{k\rightarrow\infty}d(x, \gamma_x^{n_k}(T_{n_k}))=d(x, \tilde{\gamma}_x(T^*)).
\end{align*}
It means $\tilde{\gamma}_x|_{[0, T^*]}$ is a future-directed timelike maximal geodesic up to a time reparameteriztion. We use ${\gamma}_x$ to denote the unit speed future-directed maximal geodesic after reparameteriztion. Then the internal of ${\gamma}_x$ is free of cut point. Let $\{t_{n_k}\}\subset [0,T^*]$ and $t_{n_k}$ monotone decrease to some $t^*\in(0,T^*)$. Also, set $p_{n_k}:=\tilde{\gamma}_x^{n_k}(t_{n_k})$ and $p:=\tilde{\gamma}_x(t^*)$. Because of Proposition 9.7 and Proposition 9.29 in \cite{beem_global_1981}, there exists some neighbourhood $O$ of $x$ contained in all $I^-(p_{n_k})$ such that $d(\cdot, p_{n_k})$ is smooth on $O$.

It is easy to check that there is a uniform positive lower bound to $d(\cdot, p_{n_k})$ on $O$, called $l$. Then due to the fact in Lemma \ref{lemLL}, the lower bound of the Hessian of distance function $d(\cdot, p_{n_k})$ are
\begin{align*}
	\mathbf{Hess}\,d(\cdot,p_{n_k})\geqslant
	\left\{
	\begin{array}{ll}
	-f_c(l)I,&c>0,\\
	-l^{-1}I, &c=0,\\
	-\sqrt{-c}\pi^{-1} I,&c<0,
	\end{array}
	\right.	
\end{align*}
which are obviously independent on $n_k$. In addition, for each $n_k$, we likewise define the upper support function of $u_{n_k}$ at $p_{n_k}$ on $O$ as
\begin{align*}
	u_{n_k,x}(z):=\phi(y_x^{n_k})-d(z, p_{n_k})-d(p_n, y_x^{n_k}).
\end{align*}
Then the fact $\{u_{n_k}\}$ is equi-Lipschitz and equi-semiconcave on $\bar O$ is a direct consequence of \cite[Lemma 14.20]{beem_global_1981} and Lemma \ref{lem9}.
\end{proof}
\begin{The}\label{thm:conti}
	Under the assumptions of Lemma \ref{lem:equi-lip-semi},  $u_n\rightrightarrows u_\phi$ on any compact set of $I^-[\Gamma]$, which means \eqref{A1} depends continuously on the initial function defined on $\Gamma$.
\end{The}
\begin{proof}
	Invoking Arzela-Ascoli theorem, we can get a subsequence of $\{u_{n_k}\}$  on $\bar O$ metioned in proof of Lemma \ref{lem:equi-lip-semi} which uniformly converges to some $u$ on any compact subset $\bar O$. Lemma \ref{lem:uni-semiconcave} claims that $u$ is also semiconcave with the same constant, whence it is also Lipschitz. Without loss of generality, this subsequence is still denoted by $\{u_{n_k}\}$ for the sake of brevity. Also by Lemma \ref{lem:uni-semiconcave}, for any $y$, the differentiable point of $u$, there exists $\{y_{k}\}_{k\geqslant 1}$ such that $y_{k}\to y$ and $u_{n_k}$ is differentiable at $y_{k}$. In this case, Lemma \ref{lem:uni-semiconcave} implies $\nabla u_{n_k}(y_k)\to \nabla u(y)$. Then by the continuity of Lorentz metric $g$,
	\begin{align}\label{eq:vis-limit}
		g(\nabla u(x),\nabla u(x))=\lim_{k\to\infty}g(\nabla u_{n_k}(y_k),\nabla u_{n_k}(y_k))=-1.
	\end{align}	
Similar to the proof of Theorem \ref{lem7}, we can show that $u$ is a viscosity solution to \eqref{A1} from \eqref{eq:vis-limit}.

For any fixed $x\in\Gamma$, we choose a sequence $\{x_m\}_{m\geqslant 1}\subset I^-[\Gamma]$ with $x_m\rightarrow x$. Since $u_{n_k}\rightrightarrows u$ and $\phi_{n_k}\rightrightarrows \phi$, then by the continuous extension of $u$ to $\Gamma$,
\begin{align*}
  u(x)=\lim_{m\rightarrow\infty}u(x_m)=\lim_{m\rightarrow\infty}\lim_{k\rightarrow\infty}u_{n_k}(x_m)=\lim_{k\rightarrow\infty}u_{n_k}(x)=\lim_{k\rightarrow\infty}\phi_{n_k}(x)=\phi(x).
\end{align*}Combining with the arbitrariness of $x\in\Gamma$, $u|_\Gamma=\phi$.
Due to the uniqueness of solution to \eqref{A1}, $u=u_\phi$. Note that the argument above does not depend on some fixed $O$ and all convergent subsequences of $\{u_n\}$ will converge to $u_\phi$, and it yields $u_n\rightrightarrows u_\phi$ on any compact subset of $I^-[\Gamma]$.
\end{proof}
\begin{Rem}
	All of results on well-posedness in this paper still hold true, when we consider instead the following Cauchy type problem
	\begin{align*}
		\left\{
		\begin{array}{ll}
			g(\nabla u(x),\nabla u(x))=-1, &x\in I^+[\Gamma],\\
			u(x)=\phi(x),&x\in\Gamma.
		\end{array}
        \right.
	\end{align*}
\end{Rem}
The only main difference is that the unique solution to the equation above is given by
\begin{align*}
	u_\phi(x):=-\sup_{y\in J_{\Gamma}^-(x)}\{\phi(y)+d(y,x)\},\,\,x\in I^+[\Gamma],
\end{align*}
the local semiconcavity of which ensures all the proofs go through with slight modification.

\noindent\textbf{Acknowledgements.} Xiaojun Cui and Siyao Zhu are supported by the National Natural Science Foundation of China (Grant No. 12171234).
\bibliographystyle{plain}
\bibliography{mybib}

\end{document}